\documentclass[reqno]{amsart}
\usepackage{amssymb} 
\usepackage{hyperref}

\begin{document}
\title[$\eta$-Ricci-Yamabe Solitons on Riemannian Submersions]
{$\eta$-Ricci-Yamabe Solitons on Riemannian Submersions from Riemannian manifolds }

\author[M. Danish Siddiqi and M. Akif Akyol$^{*}$ ]
{ Mohd. Danish Siddiqi and Mehmet Akif Akyol$^{*}$}

\address
{Department of Mathematics, College of Science, Jazan University, Jazan,
Kingdom of Saudi Arabia.}
\email {anallintegral@gmail, msiddiqi@jazanu.edu.sa}

\address
 {Department of Mathematics, Faculty of Arts and Sciences
Bingol University, 12000 Bingol, Turkey}
\email {mehmetakifakyol@bingol.edu.tr}



\begin{abstract}
In this research article, we establish the geometrical bearing on Riemannian submersions in terms of 
$\eta$-Ricci-Yamabe Soliton with the potential field and giving the classification of any fiber of Riemannian submersion is  an $\eta$-Ricci-Yamabe soliton, $\eta$-Ricci soliton and $\eta$-Yamabe soliton.  We also discuss the various conditions for which the target manifold of Riemannian submersion is an $\eta$-Ricci-Yamabe soliton, $\eta$-Ricci soliton,  $\eta$-Yamabe soliton and quasi-Yamabe soliton. In a particular case when the potential filed $V$ of the $\eta$-Ricci-Yamabe soliton is of gradient type, we derive a Laplacian equation and providing
 some examples of an $\eta$-Ricci-Yamabe soliton on a Riemannian submersion. Finally, we study harmonic aspect of $\eta$-Ricci-Yamabe soliton on Riemannian submersions and mention geometrical and physical effects of Ricci-Yamabe solitons.

\end{abstract}
\maketitle
\numberwithin{equation}{section}
\newtheorem{theorem}{Theorem}[section]
\newtheorem{lemma}[theorem]{Lemma}
\newtheorem{proposition}[theorem]{Proposition}
\newtheorem{corollary}[theorem]{Corollary}
\newtheorem*{remark}{Remark}
\newtheorem{Agreement}[theorem]{Agreement}
\newtheorem{definition}[theorem]{Definition}
\newtheorem{example}[theorem]{Example}

\section{Introduction} 
The notion of Riemannian immersion has been intensively studied since the very beginning of Riemannian geometry. Indeed initially the Riemmnain manifolds to be studied were surfaces imbedded in $\mathbb{R}^{3}$. In 1956, Nash \cite{Nash} proved that a revolution for Riemannian manifold that every Riemannian manifold can be isometrically embedded in any small part of Euclidean space. As a consequences, the differential geometry of Riemannian immersions well known and available in many text books.\\
\indent
 On the contrary "dual" concept of Riemannian submersion appears to have been studied and its differential geometry was first exposed by  O' Neill 1966 and  Gray 1967. We note that Riemannian submersions have been studied widely not only in mathematics, but also in theoretical pyhsics,
because of their applications in the Yang-Mills theory, Kaluza Klein theory, super gravity, relativity and
superstring theories (see \cite{BL1}, \cite{BL}, \cite{IV}, \cite{IV1}, \cite{M}, \cite{W1}). Most of the studies related to Riemannian submersion can be found in the books (\cite{Fa}, \cite{baykit}).

\indent
On the other hand, in the past twenty years the theory of geometric flows are most significant geometrical tools to explain the geometric structures in Riemannian geometry. A certain section of solutions on which the metric evolves by dilations and diffeomorphisms plays a important part in the study of singularities of the flows as they appear as possible singularity models. They are often called soliton solutions.\\
\indent
Hamilton \cite{Hami} first time introduced the concept of Ricci flow and Yamabe flow simultaneously in 1988. Ricci soliton  and Yamabe soliton  emerges as the limit of the solutions of the Ricci flow and Yamabe  flow respectively. In dimension $n=2$ the Yamabe soliton is equivalent to Ricci soliton. However, in dimension $n > 2$, the Yamabe and Ricci solitons do not agree as the first preserves the conformal class of the metric but the Ricci soliton does not in general.\\
\indent
Over the past twenty years the theory of geometric flows, such as Ricci flow and Yamabe flow has been the focus of attraction of many geometers. Recently, in 2019, Guler and Crasmareanu \cite{Guler} introduced the study of a new geometric flow which is a scalar combination of Ricci and Yamabe flow under the name Ricci-Yamabe map. This is also called Ricci-Yamabe flow of the type $(\alpha,\beta)$. The Ricci-Yamabe flow is an evolution for the metrics on the Riemannian or semi-Riemannian manifolds defined as \cite{Guler}
\begin{equation}\label{d1}
\frac{\partial}{\partial t}g(t)=-2\alpha Ric(t)+\beta R(t) g(t), \quad g_{0}=g(0).
\end{equation}
Due to the sign of involved scalars $\alpha$ and $\beta$ the Ricci-Yamabe flow can be also a Riemannian or semi-Riemannian or singular Riemannian flow. This kind of multiple choices can be useful in some geometrical or physical model for example relativistic theories.Therefore naturally Ricci-Yamabe soliton emerges as the limit of the soliton of Ricci -Yamabe flow. This is a strong inspiration for initiated the study of Ricci-Yamabe solitons is the fact that although Ricci solitons and Yamabe soliton are same in tow dimensional study, they are essentially different in higher dimensions. An interpolation solitons between Ricci and Yamabe soliton is consider in \cite{cat1} where the name Ricci-Bourguignon soliton coresponding to Ricci-Bourguignon flow but its depend on a single scalar.\\

A soliton to the Ricci-Yamabe flow is called Ricci-Yamabe solitons if it moves only by one parameter group of diffeomorphism and scaling. To be precise a Ricci-Yamabe soliton on Riemannain manifold $(M,g)$ is a data $(g,V,\lambda,\alpha,\beta)$ satisfying
\begin{equation}\label{d2}
\mathcal{L}_{V}g+2\alpha S+(2\lambda-\beta R)g=0,
\end{equation}
where $S$ is the Ricci tensor, $R$ is the scalar curvature, $\mathcal{L}_{V}$ is the Lie-derivative along the vector field. If $\lambda > 0$, $\lambda < 0$ or $\lambda=0$, then the $(M,g)$ is called \textit{Ricci-Yamabe shrinker}, \textit{Ricci-Yamabe expander} or \textit{Ricci-Yamabe steady soliton}  respectively. Therefore, equation (\ref{d2}) is called Ricci-Yamabe soliton of $(\alpha,\beta)$-type, which is a generalization of Ricci and Yamabe solitons. We note that Ricci-Yamabe soliton of type $(\alpha,0)$, $(0,\beta)$-type are $\alpha$-Ricci soliton and $\beta$-Yamabe soliton respectively.\\
\indent
An advance extension of Ricci soliton is the concept of $\eta$-Ricci soliton defined by  Cho and  Kimura \cite{cho}. Therefore analogously we can define the new notion  by 
perturbing the equation (\ref{d2}) that define the type of soliton by a multiple of a certain
$(0, 2)$-tensor field $\eta\otimes\eta$, we obtain a slightly more general notion, namely, $\eta$-Ricci-Yamabe soliton of type $(\alpha,\beta)$ defined as :
\begin{equation}\label{d3}
\mathcal{L}_{V}g+2\alpha S+(2\lambda-\beta R)g+2\mu\eta\otimes\eta=0.
\end{equation}
 Again let us remark that  $\eta$-Ricci-Yamabe soliton of type $(\alpha,0)$ or $(1,0)$, $(0,\beta)$ or$(0,1)$-type are $\alpha-\eta$-Ricci soliton (or $\eta$-Ricci soliton) and $\beta$-$\eta$-Yamabe soliton (or $\eta$-Yamabe soliton) respectively for more details about these  particular cases see ( \cite{Bla1}, \cite{Bla2}, \cite{Bla3}, \cite{Chen}, \cite{cra}, \cite{Guler1}, \cite{semsi}, \cite{Dan1}, \cite{Dan2}).\\
\indent
According to  Pigola et al.  \cite{Pigo} if  we replace the constant $\lambda$ in (\ref{d3}) with a smooth function $\lambda\in C^{\infty}(M)$, called  soliton function, then we say that $(M,g)$ is an almost Ricci soliton.  It is worth to remark that they  arise from the {\it{Ricci-Bourguignon flow}} \cite{Dan3} recently studied by G. Cantino and L. Mazzieri (\cite{cat1}, \cite{cat2}). In this more general setting, we call  (\ref{d3}) as  being fundamental equation of an {\it{almost $\eta$-Ricci-Yamabe soliton}}.

\begin{example}\label{ex1}
Let us consider the case of Einstein solitons, which generates self-similar solutions to Einstein flow such that
\end{example}
\begin{equation}\nonumber
\frac{\partial}{\partial t}g(t)=-2\left(Ric-\frac{R}{2}g\right).
\end{equation}
Therefore, an Einstein soliton emerges as the limit of the solution of Einstein flow, such that
\begin{equation}\label{d4}
\mathcal{L}_{V}g+2S+(\lambda-\frac{R}{2})g=0.
\end{equation}
In this case comparing equation (\ref{d3}) with (\ref{d4}) we have $\alpha=1$ and $\beta=\frac{1}{2}$, i,e its type $(1,\frac{1}{2})$- Ricci-Yamabe soliton.\\
\indent

The paper is organized as follows. Section 2 we give some general background on Riemannian submersions. Section 3 we recall some curvature properties of Riemannian submersions. Section 4 we define and study $\eta-$ Ricci-Yamabe soliton on Riemannian submersions from Riemannian manifolds. Section 5 we mention some examples of $\eta-$ Ricci-Yamabe soliton on Riemannian submersions. The last section of this paper we study harmonic aspect of $\eta$-Ricci-Yamabe soliton on Riemannian submersions and mention geometrical and physical effects of Ricci-Yamabe solitons.

\section{Riemannian submersions}
In this section,  we provides the necessary background for
Riemannian submersions.\\

Let $(M,g)$ and $(N,g_{\text{\tiny$N$}})$ be Riemannian manifolds,
where $dim(M)>dim(N)$. A surjective mapping
$\pi:(M,g)\rightarrow(N,g_{N})$ is called a \emph{Riemannian
submersion}
\cite{O} if:\\
\textbf{(S1)}\quad The rank of $\pi$ equals $dim(N)$.\\
In this case, for each $q\in N$, $\pi^{-1}(q)=\pi_{q}^{-1}$ is a $k$-dimensional
submanifold of $M$ and called a \emph{fiber}, where $k=dim(M)-dim(N).$
A vector field on $M$ is called \emph{vertical} (resp.
\emph{horizontal}) if it is always tangent (resp. orthogonal) to
fibers. A vector field $X$ on $M$ is called \emph{basic} if $X$ is
horizontal and $\pi$-related to a vector field $X_{*}$ on $N,$ i.e. ,
$\pi_{*}(X_{p})=X_{*\pi(p)}$ for all $p\in M,$ where $\pi_{*}$ is derivative or differential map of $\pi.$
We will denote by $\mathcal{V}$ and $\mathcal{H}$ the projections on the vertical
distribution $ker\pi_{*}$, and the horizontal distribution
$ker\pi_{*}^{\bot},$ respectively. As usual, the manifold $(M,g)$ is called \emph{total manifold} and
the manifold $(N,g_{N})$ is called \emph{base manifold} of the submersion $\pi:(M,g)\rightarrow(N,g_{N})$.\\
\textbf{(S2)}\quad $\pi_{*}$ preserves the lengths of the horizontal vectors.\\
These conditions are equivalent to say that the derivative map $\pi_{*}$ of $\pi$, restricted to $ker\pi_{*}^{\bot},$ is a linear
isometry.
If $X$ and $Y$ are the basic vector fields, $\pi$-related to $X_{N}, Y_{N}$, we have the following facts:
\begin{enumerate}
\item{} $g(X,Y)=g_{N}(X_{N},Y_{N})\circ\pi$,\\
\item{} $h[X,Y]$ is the basic vectr field $\pi$-related to $[X_{N},Y_{N}]$,\\
\item{} $h(\nabla_{X}Y)$ is the basic vector field $\pi$-related to ${\nabla^{N}}_{{X}_{N}}Y_{N}$, \\
\end{enumerate}
for any vertical vector field $\mathcal{V}$, $[X,Y]$ is the vertical.\\
\indent
The geometry of Riemannian
submersions is characterized by O'Neill's tensors $\mathcal{T}$ and
$\mathcal{A}$, defined as follows:
\begin{equation}\label{e1}
\mathcal{T}_{E}F=\mathcal{V}\nabla_{\mathcal{V}E}\mathcal{H}F+\mathcal{H}\nabla_{\mathcal{V}E}\mathcal{V}F,
\end{equation}
\begin{equation}\label{e2}
\mathcal{A}_{E}F=\mathcal{V}\nabla_{\mathcal{H}E}\mathcal{H}F+\mathcal{H}\nabla_{\mathcal{H}E}\mathcal{V}F
\end{equation}
for any vector fields $E$ and $F$ on $M,$ where $\nabla$ is the
Levi-Civita connection of $g.$ It is easy to see
that $\mathcal{T}_{E}$ and $\mathcal{A}_{E}$ are skew-symmetric
operators on the tangent bundle of $M$ reversing the vertical and
the horizontal distributions. We summarize the properties of the
tensor fields $\mathcal{T}$ and $\mathcal{A}$. Let $V,W$ be
vertical and $X,Y$ be horizontal vector fields on $M$, then we
have
\begin{equation}\label{e3}
\mathcal{T}_{V}W=\mathcal{T}_{W}V,
\end{equation}
\begin{equation}\label{e4}
\mathcal{A}_{X}Y=-\mathcal{A}_{Y}X=\frac{1}{2}\mathcal{V}[X,Y].
\end{equation}
On the other hand, from (\ref{e1}) and (\ref{e2}), we obtain
\begin{equation}\label{e5}
\nabla_{V}W=\mathcal{T}_{V}W+\hat{\nabla}_{V}W,
\end{equation}
\begin{equation}\label{e6}
\nabla_{V}X=\mathcal{T}_{V}X+\mathcal{H}\nabla_{V}X,
\end{equation}
\begin{equation}\label{e7}
\nabla_{X}V=\mathcal{A}_{X}V+\mathcal{V}\nabla_{X}V,
\end{equation}
\begin{equation}\label{e8}
\nabla_{X}Y=\mathcal{H}\nabla_{X}Y+\mathcal{A}_{X}Y,
\end{equation}
where $\hat{\nabla}_{V}W=\mathcal{V}\nabla_{V}W$.
Moreover, if $X$ is basic, then we have $\mathcal{H}\nabla_{V}X=\mathcal{A}_{X}V$. It
is not difficult to observe that $\mathcal{T}$ acts on the fibers as
the second fundamental form while $\mathcal{A}$  acts on the
horizontal distribution and measures of the obstruction to the
integrability of this distribution. For  details on the Riemannian
submersions, we refer to O'Neill's paper \cite{O} and to
the book \cite{Fa}.\\
\indent
Finally, we recall that the notion of the second fundamental form
of a map between Riemannian manifolds. Let $(M,g)$ and
$(N,g_{\text{\tiny$N$}})$ be Riemannian manifolds and
$f:(M,g)\rightarrow(N,g_{\text{\tiny$N$}})$ be a smooth map. Then,
the second fundamental form of $f$ is given by
\begin{equation}\label{e9}
\begin{array}{c}
(\nabla f_{*})(E,F)=\nabla^{f}_{E}f_{*}F-f_{*}(\nabla_{E}F)
\end{array}
\end{equation}
for $E,F\in \Gamma(TM),$ where $\nabla^{f}$ is the pull back
connection and we denote for  convenience by $\nabla$ the Riemannian
connections of the metrics $g$ and $g_{\text{\tiny$N$}}$.
It is well known that the second fundamental form is symmetric .
Moreover, $f$ is said to be \emph{totally geodesic} if $(\nabla f_{*})(E,F)=0$
for all $E,F\in \Gamma(TM)$ (see [, page 119]), and $f$ is called a \emph{harmonic} map if $trace(\nabla f_{*})=0$ (see [2, page 73]).

\section{Curvature Properties}

In this section, we have discuss some useful curvature properties of Riemannian submersion:
\begin{proposition}\label{p1}
The Riemannian curvature tensors of $(M,g)$, $(N,g_{N})$ and any fiber of $\pi$ denoting by $R$, $R^{N}$ and $\hat{R}$ respectively. Then we have
\begin{equation}\label{c1}
R(E,F,G,H)=\hat{R}(E,F,G,H)-g(\mathcal{T}_{E}H,\mathcal{T}_{F}G)+g(\mathcal{T}_{F}H,\mathcal{T}_{E}G),
\end{equation}
\begin{equation}\label{c2}
R(X,Y,Z,W)=R_{N}(X_{N},Y_{N},Z_{N},W_{N})\circ\pi+2g(\mathcal{A}_{X}Y,\mathcal{A}_{Z}W)
\end{equation}
\begin{equation}\nonumber
\quad\quad\quad-g(\mathcal{A}_{Y}Z,\mathcal{A}_{X}W)+g(\mathcal{A}_{X}Z,\mathcal{A}_{Y}W).
\end{equation}
for any $E,F,G,H\in\Gamma V(M)$ and $X,Y,Z,W\in\Gamma H(M)$.
\end{proposition}
\begin{proposition}\label{p3}
The Ricci curvature tensors of $(M,g)$, $(N,g_{N})$ and any fiber of $\pi$ denoting by $S$, $S^{N}$ and $\hat{S}$ respectively. Then we have
\begin{equation}\label{c3}
S(E,F)=\hat{S}(E,F)+g(N,\mathcal{T}_{E}F)-\sum^{n}_{i=1}g((\nabla_{X_{i}}\mathcal{T})(E,F),X_{i})-g(\mathcal{A}_{X_{i}}E,\mathcal{A}_{X_{i}}F)
\end{equation}
\begin{equation}\label{c4}
S(X,Y)=S^{N}(X^{N},Y^{N})\circ\pi-\frac{1}{2}\left\{g(\nabla_{X}N,Y)+g(\nabla_{Y}N,X)\right\},
\end{equation}
\begin{equation}\nonumber
\quad\quad\quad+2\sum^{n}_{i=1}g(\mathcal{A}_{X}{X_{i}},\mathcal{A}_{Y}{X_{i}})+\sum^{r}_{j=1}g(\mathcal{T}_{U_{i}}{X},\mathcal{T}_{U_{i}}Y),
\end{equation}
\begin{equation}\label{c5}
S(E,X)=-g(\nabla_{E}N,X)+\sum_{j}g((\nabla_{U_{i}}\mathcal{T})(E_{j},E),X)
\end{equation}
\begin{equation}\nonumber
\quad\quad\quad\quad\quad\quad\quad\quad-\sum^{n}_{i=1}\left\{g((\nabla_{X_{i}}\mathcal{A})({X_{i}},X),E)+2g(\mathcal{A}_{X_{i}}X,\mathcal{T}_{E}{X_{i}})\right\}
\end{equation}
where $\left\{X_{i}\right\}$ and $\left\{E_{i}\right\}$ are the orthonormal basis of $\mathcal{H} (horizontal)$ and $\mathcal{V} (vertical)$, respectively for any $E,F\in\Gamma V(M)$ and $X,Y\in\Gamma H(M)$.
\end{proposition}
\noindent
On the other side, for any fiber of Riemannian submersion $\pi$, the mean curvature vector field $H$ is given by $rH=N$, such that
\begin{equation}\label{c7}
N=\sum^{r}_{j=1}\mathcal{T}_{E_{j}}E_{j}
\end{equation}
and also the dimension of any fiber of $\pi$ denotes by $r$ and $\left\{E_{1},E_{2},...E_{r}\right\}$ is an orthonormal basis on vertical distribution. We point that the horizontal vector field $N$ vanishes if and only if any fiber of Riemannian submersion $\pi$ is minimal.\\
\indent
Now, from (\ref{c7}), we find
\begin{equation}\label{c8}
g(\nabla_{U}N,X)=\sum^{r}_{j=1}g((\nabla_{U}\mathcal{T})(E_{j},E_{j}),X)
\end{equation}
for any $U\in\Gamma(TM)$ and $X\in\Gamma H(M)$.\\
\noindent
Horizontal divergence of any vector field $X$ on $\Gamma H(M)$ denoted by $div(X)$ and given by
\begin{equation}\label{cc}
div(X)=\sum^{n}_{i=1}g(\nabla_{X_{i}}X,X_{{i}}),
\end{equation}
where $\left\{X_{1},X_{2},....X_{n}\right\}$ is an orthonormal basis of horizontal space $\Gamma H(M)$. Hence, considering (\ref{cc}), we have
\begin{equation}\label{cc1}
div(N)=\sum^{n}_{i=1}\sum^{r}_{j=1}g(\nabla_{X_{i}}\mathcal{T})(E_{j},E_{j}),X_{i}).
\end{equation}
\section{$\eta$-Ricci-Yamabe Soliton on Riemannian submersions}

This section deal with the study of $\eta$-Ricci-Yamabe soliton of type $(\alpha,\beta)$ on Riemannian submersion $\pi:(M,g)\longrightarrow(N,g_{N})$ from Riemannian manifold and discussed the nature of fiber of such submersion with target manifold $(N,g_{N}$):\\
\indent
As a consequences of equations (\ref{e5}) to (\ref{e8}) in case of Riemannian submersions, we find the following results:

\begin{theorem}\label{t1}
Let $\pi$ be a Riemannian submersion between Riemannian manifolds. Then , the vertical distribution $\mathcal{V}$ is parallel with respect to the connection $\nabla$, if the horizontal parts $\mathcal{T}_{F}H$ and $\mathcal{A}_{X}F$ of (\ref{e5}) and (\ref{e7}) vanish, identically. Similarly, the horizontal distribution $\mathcal{H}$ is parallel with respect to the connection $\nabla$, if the vertical parts $\mathcal{T}_{F}X$ and $\mathcal{A}_{X}Y$ of (\ref{e6}) and (\ref{e8}) vanish, identically, for any $X,Y\in\Gamma H(M)$ and $F,H\in\Gamma V(M)$.
\end{theorem}
\indent
Since $(M,g)$ is an $\eta$-Ricci-Yamabe soliton, then from (\ref{d3}) we have
\begin{equation}\label{r1}
(\mathcal{L}_{V}g)(E,F)+2\alpha S(E,F)+(2\lambda-\beta R)g(E,F)+2\mu\eta(E)\eta(F)=0
\end{equation}
for any $E,F\in\Gamma V(M)$. Using the equation (\ref{c3}), we have 
\begin{equation}\label{r2}
\left\{g(\nabla_{E}V,F)+g(\nabla_{F}V,E)\right\}+2\alpha \hat{S}(E,F)+g(N,\mathcal{T}_EF)
\end{equation}
\begin{equation}\nonumber
-\sum^{n}_{i=1}g((\nabla_{X_{i}}\mathcal{T})(E,F),X_{i})-g(\mathcal{A}_{X_{i}}E,\mathcal{A}_{X_{i}}F)+(2\lambda-\beta R)g(E,F)+2\mu\eta(E)\eta(F)=0
\end{equation}
where $\left\{X_{i}\right\}$ denotes an orthonormal basis of the horizontal distribution $\mathcal{H}$ and $\nabla$ is the Levi-Civita connection on $M$. Then using Theorem \ref{t1}, the equations (\ref{e2}) and (\ref{e5}), we find the following equation
\begin{equation}\label{r3}
[\hat{g}(\hat{\nabla}_{E}V,F)+\hat{g}(\hat{\nabla}_{F}V,E)]+2\alpha \hat{S}(E,F)+(2\lambda-\beta\hat {R})\hat{g}(E,F)+2\mu\eta(E)\eta(F)=0,
\end{equation}
for any $E,F\in\Gamma V(M)$. Thus we have the following results:
\begin{theorem}\label{t2}
Let $(M,g,V,\lambda,\mu,\alpha,\beta)$ be an $\eta$-Ricci-Yamabe soliton of type $(\alpha,\beta)$ with vertical potential field $V$ and $\pi$ be a Riemannian submersion between Riemannian manifolds. If the vertical distribution $\mathcal{V}$ is parallel, then any fiber of Riemannian submersion $\pi$ is an $\eta$-Ricci-Yamabe soliton.
\end{theorem}

\begin{remark}\label{R1}
Now, for $\alpha=1, \beta=0$ and $\mu\neq 0$, then from  equation (\ref{r3}) we find
\begin{equation}\label{r4}
[\hat{g}(\hat{\nabla}_{E}V,F)+\hat{g}(\hat{\nabla}_{F}V,E)]+2\hat{S}(E,F)+2\lambda\hat{g}(E,F)+2\mu\eta(E)\eta(F)=0,
\end{equation}
\end{remark} 
 Thus we have the following result:
\begin{theorem}\label{T3}
Let $(M,g,V,\lambda,\mu)$ be an $\eta$-Ricci soliton of type $(1,0)$ with vertical potential field $V$ and $\pi$ be a Riemannian submersion between Riemannian manifolds. If the vertical distribution $\mathcal{V}$ is parallel, then any fiber of Riemannian submersions $\pi$ is an $\eta$-Ricci soliton.
\end{theorem}

\begin{remark}\label{R2}
Now, for $\alpha=0, \beta=1$ and $\mu=0 $, then from  equation (\ref{r3}) we find
\begin{equation}\label{r5}
[\hat{g}(\hat{\nabla}_{E}V,F)+\hat{g}(\hat{\nabla}_{F}V,E)]+2\hat{S}(E,F)+2\lambda \hat{g}(E,F)=0,
\end{equation}
\end{remark}
Thus we have the following result:
\begin{theorem}\label{t4}
Let $(M,g,V,\lambda )$ be an Ricci soliton of type $(1,0)$ and $\mu=0$ with vertical potential field $V$ and $\pi$ be a Riemannian submersion between Riemannian manifolds. If the vertical distribution $\mathcal{V}$ is parallel, then any fiber of Riemannian submersions $\pi$ is a Ricci soliton.
\end{theorem}
\begin{remark}\label{R3}
Now, for $\alpha=0, \beta=1$ and $\mu\neq 0$, then from  equation (\ref{r3}) we find
\begin{equation}\label{r6}
[\hat{g}(\hat{\nabla}_{E}V,F)+\hat{g}(\hat{\nabla}_{F}V,E)]+(2\lambda- \hat{R})\hat{g}(E,F)+2\mu\eta(E)\eta(F)=0,
\end{equation}
\end{remark}
 Thus we have the following result:
\begin{theorem}\label{T5}
Let $(M,g,V,\lambda,\mu)$ be an $\eta$-Ricci soliton of type $(0,1)$ with vertical potential field $V$ and $\pi$ be a Riemannian submersion between Riemannian manifolds. If the vertical distribution $\mathcal{V}$ is parallel, then any fiber of Riemannian submersions $\pi$ is an $\eta$-Yamabe soliton.
\end{theorem}

\begin{remark}\label{R4}
Now, for $\alpha=0, \beta=1$ and $\mu= 0$, then from  equation (\ref{r3}) we find
\end{remark}
\begin{equation}\label{r7}
[\hat{g}(\hat{\nabla}_{E}V,F)+\hat{g}(\hat{\nabla}_{F}V,E)]+(2\lambda- \hat{R})\hat{g}(E,F)=0,
\end{equation}
 Thus we have the following results:
\begin{theorem}\label{t6}
Let $(M,g,V,\lambda,\mu)$ be an $\eta$-Ricci soliton of type $(0,1)$ and $\mu=0$ with vertical potential field $V$ and $\pi$ be a Riemannian submersion between Riemannian manifolds. If the vertical distribution $\mathcal{V}$ is parallel, then any fiber of Riemannian submersions $\pi$ is a quasi-Yamabe soliton.
\end{theorem}
Since the total space $(M,g)$ of Riemannian submersion $\pi:(M,g)\longrightarrow(N,g_{N})$ admits an $\eta$-Ricci-Yamabe soliton of type $(\alpha,\beta)$, now from equations (\ref{d3}) and (\ref{c3}), we find 
\begin{equation}\label{r8}
\left\{g(\nabla_{E}V,F)+g(\nabla_{F}V,E)\right\}+2\alpha \hat{S}(E,F)+\sum^{r}_{j=1}g(\mathcal{T}_{E_{j}}E_{j},\mathcal{T}_{E}F)
\end{equation}
\begin{equation}\nonumber
-\sum^{n}_{i=1}g((\nabla_{X_{i}}\mathcal{T})(E,F),X_{i})-g(\mathcal{A}_{X_{i}}E,\mathcal{A}_{X_{i}}F)+(2\lambda-\beta \hat{R})\hat{g}(E,F)+2\mu\eta(E)\eta(F)=0
\end{equation}
for any $E,F\in\Gamma V(M)$. Also, the $\eta$-Ricci-Yamabe soliton $(M,g,V,\lambda,\mu)$ of $(\alpha,\beta)$ type has totally umbilical fibers and using equation (\ref{e5}) in equation (\ref{r8}), we find 
\begin{equation}\label{r9}
\left\{g(\hat{\nabla}_{E}V,H)+g(\hat{\nabla}_{H}V,E)\right\}+2\alpha \hat{S}(E,H)+\sum^{r}_{j=1}g(\mathcal{T}_{E_{j}}E_{j},\mathcal{T}_{E}H)
\end{equation}
\begin{equation}\nonumber
-\sum^{n}_{i=1}\left\{(\nabla_{X_{i}}g)(E,H)g(W,X_{i})-g(\nabla_{X_{i}}W,X_{i})\hat{g}(E,H)\right\}
\end{equation}
\begin{equation}\nonumber
-\sum^{n}_{i=1}g(\mathcal{A}_{X_{i}}E,\mathcal{A}_{X_{i}}H)+(2\lambda-\beta \hat{R})\hat{g}(E,H)+2\mu\eta(E)\eta(H)=0.
\end{equation}

Since the horizontal distribution $\mathcal{H}$ is integrable, we have
\begin{equation}\label{r10}
(\mathcal{L}_{V}\hat{g})(E,H)+2\alpha\hat{S}(E,H)-\sum^{n}_{i=1}g(\nabla_{X_{i}}W,X_{{i}})\hat{g}(E,H)
\end{equation}
\begin{equation}\nonumber
\quad\quad\quad\quad\quad+r\left\|W\right\|^{2}\hat{g}(E,H)+(2\lambda-\beta \hat{R})\hat{g}(E,H)+2\mu\eta(E)\eta(H)=0
\end{equation}
where $W$ is the mean curvature vector of any fiber of $\pi$. From (\ref{cc}), we find
\begin{equation}\label{r11}
(\mathcal{L}_{V}\hat{g})(E,H)+2\alpha\hat{S}(E,H)+[r\left\|W\right\|^{2}-div(W)+2\lambda-\beta \hat{R}]\hat{g}(E,H)+2\mu\eta(E)\eta(H)=0
\end{equation}
which shows that any fiber of $\pi$ is an almost $\eta$-Ricci-Yamabe soliton. Thus we can state the following result:

\begin{theorem}\label{t3}
Let $(M,g,\lambda,\mu,\alpha,\beta)$ be an $\eta$-Ricci-Yamabe soliton of type-$(\alpha,\beta)$ with the vertical potential field $V$ and $\pi$ be a Riemannian submersion between Riemannian manifolds with totally umbilical fibers. If the horizontal distribution $\mathcal{H}$ is integrable, then any fiber of Riemannian submersion $\pi$ is an almost $\eta$-Ricci-Yamabe soliton.
\end{theorem}
Now, we also have the following results:
\begin{theorem}\label{t4}
Let $(M,g,\lambda,\mu,\alpha,\beta)$ be an $\eta$-Ricci-Yamabe soliton of type-$(1,0)$ with the vertical potential field $V$ and $\pi$ be a Riemannian submersion between Riemannian manifold with totally umbilical fibers. If the horizontal distribution $\mathcal{H}$ is integrable, then any fiber of Riemannian submersion $\pi$ is an almost $\eta$-Ricci soliton.
\end{theorem}
\begin{proof}
For $\alpha=1,\beta=0$, $\mu\neq 0$ and from (\ref{r11}) we obtain the required result.
\end{proof}

\begin{theorem}\label{t5}
Let $(M,g,\lambda,\mu,\alpha,\beta)$ be an $\eta$-Ricci-Yamabe soliton of type-$(0,1)$ with the vertical potential field $V$ and $\pi$ be a Riemannian submersion between Riemannian manifold with totally umbilical fibers. If the horizontal distribution $\mathcal{H}$ is integrable, then any fiber of Riemannian submersion $\pi$ is an almost $\eta$-quasi Yamabe soliton.
\end{theorem}
\begin{proof}
For $\alpha=0,\beta=1$, $\mu\neq 0$ and using (\ref{r11}) we obtain the required result.
\end{proof}
\indent
For $\mu=0$, we also have the following corollaries as a consequences of Theorem (\ref{t4}) and Theorem (\ref{t5}).
\begin{corollary}\label{t6}
Let $(M,g,\lambda,\mu,\alpha,\beta)$ be an $\eta$-Ricci-Yamabe soliton of type-$(1,0)$ with the vertical potential field $V$ and $\pi$ be a Riemannian submersion between Riemannian manifold with totally umbilical fibers. If the horizontal distribution $\mathcal{H}$ is integrable, then any fiber of Riemannian submersion $\pi$ is an almost Ricci soliton.
\end{corollary}

\begin{corollary}\label{t7}
Let $(M,g,\lambda,\mu,\alpha,\beta)$ be an $\eta$-Ricci-Yamabe soliton of type-$(0,1)$ with the vertical potential field $V$ and $\pi$ be a Riemannian submersion between Riemannian manifold with totally umbilical fibers. If the horizontal distribution $\mathcal{H}$ is integrable, then any fiber of Riemannian submersion $\pi$ is an almost qusi-Yamabe soliton.
\end{corollary}
\indent
Again, assuming the Theorem (\ref{t3}), we obtain the following  of corollaries:

\begin{corollary}\label{t8}
Let $(M,g,\lambda,\mu,\alpha,\beta)$ be an $\eta$-Ricci-Yamabe soliton of type-$(\alpha,\beta)$ and $\pi$ be a Riemannian submersion between Riemannian manifolds, such that the horizontal distribution $\mathcal{H}$ is integrable, then any fiber of Riemannian submersion $\pi$ is an almost $\eta$-Ricci-Yamabe soliton, if any fiber of $\pi$ is totally umbilical and has a constant mean curvature.
\end{corollary}
\begin{corollary}\label{t9}
Let $(M,g,\lambda,\mu,\alpha,\beta)$ be an $\eta$-Ricci-Yamabe soliton of type-$(\alpha,\beta)$ and $\pi$ be a Riemannian submersion between Riemannian manifolds, such that the horizontal distribution $\mathcal{H}$ is integrable, then any fiber of Riemannian submersion $\pi$ is an almost $\eta$-Ricci-Yamabe soliton, if any fiber of $\pi$ is totally geodesic.
\end{corollary}

\begin{corollary}\label{t10}
Let $(M,g,\lambda,\mu,\alpha,\beta)$ be an $\eta$-Ricci-Yamabe soliton of type-$(1,0)$ and $\pi$ be a Riemannian submersion between Riemannian manifolds, such that the horizontal distribution $\mathcal{H}$ is integrable, then any fiber of Riemannian submersion $\pi$ is an almost $\eta$-Ricci soliton, if any fiber of $\pi$ is totally umbilical and has a constant mean curvature.
\end{corollary}
\begin{corollary}\label{t11}
Let $(M,g,\lambda,\mu,\alpha,\beta)$ be an $\eta$-Ricci-Yamabe soliton of type-$(0,1)$ and $\pi$ be a Riemannian submersion between Riemannian manifolds, such that the horizontal distribution $\mathcal{H}$ is integrable, then any fiber of Riemannian submersion $\pi$ is an almost $\eta$-quasi Yamabe soliton, if any fiber of $\pi$ is totally geodesic.
\end{corollary}
\indent
Then, we have the following theorem:
\begin{theorem}\label{t12}
Let $(M,g,U,\lambda,\mu,\alpha,\beta)$ be an $\eta$-Ricci-Yamabe soliton of type-$(\alpha, \beta)$ with the potential field $U\in\Gamma(TM)$ and $\pi$ be a Riemannian submersion between Riemannian manifolds. If the horizontal distribution $\mathcal{H}$ is parallel, then the following are satisfied:
\begin{enumerate}
\item{} If the vector field $U$ is vertical, then $(N,g_{N})$ is an $\eta$-Einstein manifold.
\item{} If the vector field $U$ is horizontal, then $(N,g_{N})$ is an $\eta$-Ricci-Yamabe soliton with potential field $U_{N}$, such that $\pi_{\ast}U=U_{N}$.
\end{enumerate}
\end{theorem}
\begin{proof}
Since the total space $(M,g)$ of Riemannian submersion $\pi$ admits a $\eta$-Ricci-Yamabe soliton of type $(\alpha,\beta)$ with potential field $U\in\Gamma(TM)$, then using (\ref{d3}) and (\ref{c4}), we have
\begin{equation}\label{s1}
[g(\nabla_{X}U,Y)+g(\nabla_{Y}U,X)]+2\alpha S_{N}(X_{N},Y_{N})\circ\pi-(g(\nabla_{X}N,Y)+g(\nabla_{Y}N,X))
\end{equation}
\begin{equation}\nonumber
+2\sum^{n}_{i=1}g(\mathcal{A}_{X}X_{{i}},\mathcal{A}_{Y}X_{i})+\sum^{r}_{j=1}g(\mathcal{T}_{E_{j}}X,\mathcal{T}_{E_{j}}Y)+(2\lambda-\beta R)g(X,Y)+2\mu\eta(X)\eta(Y)=0
\end{equation}
where $X_{N}$ and $Y_{N}$ are $\pi$-related to $X$ and $Y$ respectively, for any $X,Y\in\Gamma H(M)$.\\
\indent
Applying Theorem (\ref{t1}) to above equation (\ref{s1}), we have
\begin{equation}\label{s2}
[g(\nabla_{X}U,Y)+g(\nabla_{Y}U,X)]+2\alpha S_{N}(X_{N},Y_{N})\circ\pi+(2\lambda-\beta R)g(X,Y)+2\mu\eta(X)\eta(Y)=0.
\end{equation}
{{\bf 1}}. \quad If vector field $U$ is vertical, from the equation (\ref{e7}), it follows
\begin{equation}\label{s3}
[g(\mathcal{A}_{X}U,Y)+g(\mathcal{A}_{Y}U,X)]+2\alpha S_{N}(X_{N},Y_{N})\circ\pi+(2\lambda-\beta R)g(X,Y)+2\mu\eta(X)\eta(Y)=0.
\end{equation}
Since $\mathcal{H}$ is parallel, we get
\begin{equation}\label{s4}
 S_{N}(X_{N},Y_{N})\circ\pi=ag(X,Y)+b\eta(X)\eta(Y)=0.
\end{equation}
which shows that $(N,g_{N})$ is an $\eta$-Einstein, where $a=-(\lambda-\frac{R}{2})$ and $b=-\mu$.\\

\noindent
{{\bf 2.}}\quad If the vector field $U$ is horizontal , then equation (\ref{s2}) becomes
\begin{equation}\label{s5}
(\mathcal{L}_{U}g)(X,Y)+2\alpha S_{N}(X_{N},Y_{N})\circ\pi+(2\lambda-\beta{R})g(X,Y)+2\mu\eta(X)\eta(Y)=0.
\end{equation}
which shows that the Riemannian manifold $(N,g_{N})$ is an $\eta$-Ricci-Yamabe soliton with horizontal potential field $E_{N}$.
\end{proof}
\indent
Now, from (\ref{s5}) and using the fact that if the vector field $U$ is horizontal , then we have following theorem:

\begin{theorem}\label{t12}
Let $(M,g,U,\lambda,\mu,\beta)$ be an $\eta$-Ricci-Yamabe soliton of type-$(0,\beta)$ with the potential field $U\in\Gamma(TM)$ and $\pi$ be a Riemannian submersion between Riemannian manifolds. If the horizontal distribution $\mathcal{H}$ is parallel and the vector field $U$ is horizontal, then $(N,g_{N})$ is a $\eta$-quasi-Yamabe soliton with horizontal potential field $E_{N},$ such that
 \begin{equation}\label{s6}
(\mathcal{L}_{U}g)(X,Y)+(2\lambda-\beta R)g(X,Y)+2\mu\eta(X)\eta(Y)=0.
\end{equation}
\end{theorem}

\indent
Again using Theorem (\ref{t1}) and equation (\ref{c4}) together, we obtain the following one:
\begin{lemma}\label{L1}
Let $(M,g,\xi,\lambda,\mu,\alpha,\beta)$ be an $\eta$-Ricci-Yamabe soliton on Riemannian submersion $\pi$ between Riemannian manifolds with horizontal potential field $\xi$, such that $\mathcal{H}$ is parallel. Then, the vector field $N$ is Killing on the horizontal distribution $\mathcal{H}$.
\end{lemma}

\indent
Since $(M,g,\xi,\lambda,\mu)$ is an $\eta$-Ricci-Yamabe soliton of type $(\alpha,\beta)$ and again using (\ref{c4}) in (\ref{d3}), we find that
\begin{equation}\label{s0}
(\mathcal{L}_{\xi}g)(X,Y)+2\alpha S^{N}(X_{N},Y_{N})\circ\pi-\left\{g(\nabla_{X}N,Y)+g(\nabla_{Y}N,X)\right\}
\end{equation}
\begin{equation}\nonumber
+2\sum_{i}g(\mathcal{A}_{X}X_{i},\mathcal{A}_{Y}X_{i})+\sum_{j}g(\mathcal{T}_{U_{j}}X,\mathcal{T}_{U_{j}}Y)+(2\lambda-\beta R)g(X,Y)+2\mu\eta(X)\eta(Y)=0.
\end{equation}

where $\left\{X_{i}\right\}$ denotes an orthonormal basis of $\mathcal{H}$, for any $X,Y\in\Gamma H(M)$. From Theorem (\ref{t1}), equation (\ref{s0}) becomes as

\begin{equation}\label{s01}
(\mathcal{L}_{\xi}g)(X,Y)+2\alpha S^{N}(X_{N},Y_{N})\circ\pi+(2\lambda-\beta R)g(X,Y)+2\mu\eta(X)\eta(Y)=0.
\end{equation}
since the Riemannian manifold $(N,g_{N})$ is an $\eta$-Einstein, we can find tha $\xi$ is conformal Killing. Thus we can state the following result:
\begin{theorem}\label{tL}
Let $(M,g,\xi,\lambda,\mu,\alpha,\beta)$ be an $\eta$-Ricci-Yamabe soliton of type $(\alpha,\beta)$ on Riemannian submersion $\pi$ from Riemannian manifold to an $\eta$-Einstein manifold with horizontal potential field $\xi$, such that horizontal distribution $\mathcal{H}$ is parallel. Then, the vector field $\xi$ is conformal Killing on the horizontal distribution $\mathcal{H}$.
\end{theorem}

\section{Examples}
\begin{example} 
Let $M^{6}= \left\{(x_{1}, x_{2},x_{3},x_{4},x_{5}, x_{6})| x_{6}\neq 0\right\}$ be a six-dimensional differentiable manifold where $(x_{i}), where \quad i=1,2,3,4,5,6$ denotes the standard coordinates of a point in $\mathbb{R}^{6}$. 
\end{example}
Let
\begin{equation}\nonumber
E_{1}=\partial x_{1},\quad E_{2}=\partial x_{2},\quad E_{3}=\partial x_{3}, \quad E_{4}=\partial x_{4},\quad E_{5}=\partial x_{5}, \quad E_{6}=\partial x_{6}.
\end{equation}

is a set of linearly independent vector fields at each point of the manifold $M^{6}$ and therefore it forms a basis for the tangent space $T(M^{6})$. We define a positive definite metric $g$ on $M^{6}$ as
\noindent
where $i,j=1,2,3,4,5,6$ and it is given by $g=\sum^{6}_{i.j=1}dx_{i}\otimes dx_{j}$. Let the $1$-form $\eta$ be defined by $\eta(X)=g(X,P)$ where $p=E_{6}$.\\
\indent
Then it is obvious the $(M^{6},g)$ is a Riemannian manifold of dimension 6. Moreover, let $\bar{\nabla}$ be the Levi-Civita connection with respect to metric $g$. Then we have
$[E_{1},E_{2}]=0$. Similarly $[E_1, E_{6}]=E_1$,~$[E_2, E_{6}]=E_2$,\quad $[E_3,E_{6}]=E_3$,~$[E_4,E_{6}]=E_4$,\quad $[E_5,E_{6}]=E_6$~$[E_{i},E_{j}]=0$, $1\leq i \neq j \leq 5$.\\
The Riemannian connection ${\hat{\nabla}}$ of the metric $\hat{g}$ is given by
\begin{align}\nonumber
2g(\hat{\nabla}_{X}Y,Z)&=Xg(Y,Z)+Yg(Z,X)-Zg(X,Y)\\
&-g(X,[Y,Z])-g(Y,[X,Z])+g(Z,[X,Y]),\nonumber
\end{align}
where $\nabla$ denotes the Levi-Civita connection corresponding to the metric $g$.\\
By using Koszul's formula and (\ref{e5}) together, we obtain the following equations
\begin{equation}\label{x1}
\hat{\nabla}_{E_{1}}E_{1}=E_{6},~ \hat{\nabla}_{E_{2}}E_{2}=E_{6},~\hat{\nabla}_{E_{3}}E_{3}=E_{6},~\hat{\nabla}_{E_{4}}E_{4}=E_{6}, \hat{\nabla}_{E_{5}}E_{5}=E_{6}
\end{equation}
\begin{equation}\nonumber
\hat{\nabla}_{E_{6}}E_{6}=0, ~\hat{\nabla}_{E_{6}}E_{i}=0,~ \hat{\nabla}_{E_{i}}E_{6}=E_{i}, ~ 1\leq i \leq 5
\end{equation}
and $\hat{\nabla}_{E_{i}}E_{i}=0$ for all $1\leq i$, $j\leq 5$. Now, from equation (\ref{c1}) and (\ref{x1})  we can calculated the non vanishing components  of Riemannian curvature tensor $\hat{R}$, Ricci curvature tensor $\hat{S}$ and scalar curvature $\hat{R}$ of fiber as 
\begin{equation}\label{x2}
\hat{R}(E_{1},E_{2})E_{1}=E_{2},\quad\hat{R}(E_{1},E_{2})E_{2}=-E_{1},\quad\hat{R}(E_{1},E_{3})E_{1}=-E_{3},\hat{R}(E_{1},E_{3})E_{3}=E_{1}
\end{equation} 
\begin{equation}\nonumber 
\hat{R}(E_{1},E_{4})E_{1}=-E_{4},\quad\hat{R}(E_{1},E_{4})E_{4}=E_{1},\quad\hat{R}(E_{1},E_{5})E_{1}=-E_{5},\hat{R}(E_{1},E_{5})E_{5}=E_{1}
\end{equation}
\begin{equation}\nonumber 
\hat{R}(E_{1},E_{6})E_{1}=-E_{6},\quad\hat{R}(E_{1},E_{6})E_{6}=-E_{1},\quad\hat{R}(E_{2},E_{3})E_{2}=-E_{3},\hat{R}(E_{2},E_{3})E_{3}=E_{2}
\end{equation}
\begin{equation}\nonumber 
\hat{R}(E_{2},E_{4})E_{2}=E_{4},\quad\hat{R}(E_{2},E_{4})E_{4}=-E_{2},\quad\hat{R}(E_{2},E_{5})E_{2}=E_{5},\hat{R}(E_{2},E_{5})E_{5}=-E_{2}
\end{equation}
\begin{equation}\nonumber 
\hat{R}(E_{2},E_{6})E_{2}=E_{6},\quad\hat{R}(E_{2},E_{6})E_{6}=-E_{2},\quad\hat{R}(E_{3},E_{4})E_{3}=E_{4},\hat{R}(E_{3},E_{4})E_{4}=E_{5}
\end{equation}
\begin{equation}\nonumber 
\hat{R}(E_{3},E_{5})E_{5}=-E_{3},\quad\hat{R}(E_{3},E_{6})E_{3}=-E_{6},\quad\hat{R}(E_{3},E_{6})E_{3}=-E_{6},\hat{R}(E_{3},E_{6})E_{6}=-E_{3}
\end{equation}
\begin{equation}\nonumber 
\hat{R}(E_{4},E_{5})E_{4}=E_{5},\quad\hat{R}(E_{4},E_{5})E_{5}=-E_{4},\quad\hat{R}(E_{4},E_{6})E_{4}=-E_{6},
\end{equation}
\begin{equation}\nonumber 
\hat{R}(E_{4},E_{6})E_{6}=-E_{4},\quad \hat{R}(E_{5},E_{6})E_{5}=-E_{6},\quad\hat{R}(E_{5},E_{6})E_{6}=-E_{5}.
\end{equation}
\[
\hat{S}(E_{i},E_{j})=
\begin{bmatrix}
	-3&0&0&0&0&0\\
0&-3&0&0&0&0\\
0&0&-3&0&0&0\\
0&0&0&-3&0&0\\
0&0&0&0&-3&0\\
0&0&0&0&0&-5\\        
\end{bmatrix}
.
\]

\begin{equation}\label{x5}
\hat{R}=Trace (\hat{S})=-20.
\end{equation}
From equation (\ref{c3}), we have
\begin{equation}\label{x6}
[\hat{g}(\hat{\nabla}_{E_{i}}E_{6},E_{i})+\hat{g}(\hat{\nabla}_{E_{i}}E_{6},E_{i})]+2\alpha \hat{S}(E_{i},E_{i})+(2\lambda-\beta\hat {R})\hat{g}(E_{i},E_{i})+2\mu\delta^{i}_{j}=0
\end{equation}
for all $i\in \left\{1,2,3,4,5,6\right\}$. Therefore $\lambda=10\beta-3\alpha-1$ and $\mu=23\alpha-30\beta-1$ the data $(\hat{g},E_{6},\lambda,\mu,\alpha,\beta)$ is an $\eta$-Ricci-Yamabe soliton, verified equation (\ref{c3}). Therefore the data $(V,\hat{g},\lambda, \mu, \alpha, \beta)$ is admitting the expanding, shrinking and steady $\eta$-Ricci-Yamabe soliton according $(3\alpha+1)> 10\beta$,  $(3\alpha+1) < 10\beta$ or $(3\alpha+1)= 10\beta$ respectively.\\

\indent
Now, we have following two main cases for particular value of $\alpha$ and $\beta$:\\

\noindent
{\bf {Case 1}}.\quad In an $\eta$-Ricci-Yamabe soliton of type $(\alpha,\beta)$ for particular  $\alpha=1, \beta=0$, we find $\lambda=-4$ and $\mu=22$, then we have the data $(\hat{g},E_{6},\lambda,\mu,1,0)$ is an $\eta$-Ricci soliton is shrinking. This case verifying Theorem (\ref{T3}).\\

\noindent
{\bf {Case 2}}.\quad In an $\eta$-Ricci-Yamabe soliton of type $(\alpha,\beta)$ for particular  $\alpha=0, \beta=0$, we find $\lambda=9$ and $\mu=-31$, then we have the data $(\hat{g},E_{6},\lambda,\mu,0,1)$ is an $\eta$-Yamabe soliton is expanding. This case verifying Theorem (\ref{T5}).\\

\begin{example}
	Let $\psi : \mathbb{R}^{6} \rightarrow \mathbb{R}^{3}$  be a submersion defined by $$\psi(x_{1},x_{2},...x_{6})=(y_1,y_2,y_3),$$ where 
	$$y_1=\frac{x_1+x_2}{\sqrt{2}}, \ y_2=\frac{x_3+x_4}{\sqrt{2}} \ \textrm{and} \ y_3=\frac{x_5+x_6}{\sqrt{2}}.$$ 
	\end{example}
	Then, the Jacobian matrix of $\psi$ has rank 3. That means $\psi$ is a submersion. A straight computations yields
		\begin{align*}
	ker\psi_{*}=span \{&V_{1}=\frac{1}{\sqrt{2}}(-\partial x_{1}+\partial x_{2}), V_{2}=\frac{1}{\sqrt{2}}(-\partial x_{3}+\partial x_{4}), \\
	&V_3=\frac{1}{\sqrt{2}}(-\partial x_{5}+\partial x_{6})\},
	\end{align*}
	and
		\begin{align*}
	(ker\psi_{*})^\perp=span \{&H_{1}=\frac{1}{\sqrt{2}}(\partial x_{1}+\partial x_{2}), H_{2}=\frac{1}{\sqrt{2}}(\partial x_{3}+\partial x_{4}), \\
	&H_3=\frac{1}{\sqrt{2}}(\partial x_{5}+\partial x_{6})\},
	\end{align*}
	Also by direct computations yields
	$$\psi_{*}(H_1)=\partial y_1, \psi_{*}(H_2)=\partial y_2 \ \textrm{and}\ \psi_{*}(H_3)=\partial y_3$$
	Hence, it is easy to see that
	$$g_{\mathbb{R}^{6}}(H_i,H_i)=g_{\mathbb{R}^{3}}(\psi_{*}(H_i),\psi_{*}(H_i)), \ i=1,2,3$$
	Thus $\psi$ is a Riemannian submersion.\\
	\indent
	
	Now, we can compute the components of Riemannian curvature tensor $\hat{R}$, Ricci curvature tensor $\hat{S}$ and scalar curvature $\hat{R}$ for $ker\psi_{\ast}$ (vertical space) and $\ker\psi_{\ast}$ (horizontal space),  respectively. For the vertical space, we have
	\begin{equation}\label{k1}
	\hat{R}(V_{1},V_{2})V_{1}=-2V_{2}, \quad \hat{R}(V_{1},V_{2})V_{2}=2V_{1},\quad \hat{R}(V_{1},V_{3})V_{1}=-2V_{3}
	\end{equation}
	\begin{equation}\nonumber
	\hat{R}(V_{1},V_{2})V_{3}=V_{1}, \quad \hat{R}(V_{2},V_{3})V_{3}=V_{2},\quad \hat{R}(V_{2},V_{3})V_{2}=V_{2}.
	\end{equation}

\[
\hat{S}(V_{i},V_{j})= 
\begin{bmatrix}
	2&0&0\\
0&2&0\\
0&0&1\\       
\end{bmatrix}
.
\]

\begin{equation}\label{k3}
\hat{R}=Trace (\hat{S})=5.
\end{equation}
Using (\ref{d3}), we obtain  $\lambda=\frac{5\beta}{{2}}-\alpha$ and $\mu =\alpha$. Therefore ($ker\psi_{\ast},g)$ is admitting the expanding, shrinking and steady $\eta$-Ricci Yamabe soliton according $\frac{5\beta}{{2}}>\alpha$,  $\frac{5\beta}{{2}}<\alpha$   or $\frac{5\beta}{{2}}=\alpha$ respectively.\\

Moreover, we also have the following cases for particular values of $\alpha$ and $\beta$ such as:\\
\noindent
{\bf {Case 1}}.\quad In an $\eta$-Ricci-Yamabe soliton of type $(\alpha,\beta)$ for particular value of  $\alpha=1, \beta=0$, we find $\lambda=-2$ and $\mu=1$, then ($ker\psi_{\ast},g)$ admitting shrinking $\eta$-Ricci soliton.\\

\noindent
{\bf {Case 2}}.\quad In an $\eta$-Ricci-Yamabe soliton of type $(\alpha,\beta)$ for particular value of $\alpha=0, \beta=1$, we find $\lambda=\frac{5}{2}$ and $\mu=0$, then we have then ($ker\psi_{\ast},g)$ admitting a expanding Yamabe soliton.\\

In a similar way, for the horizontal space, we have
\begin{align*}
	&R^{N}(\psi_{*}(H_1),\psi_{*}(H_2))\psi_{*}(H_1)=\frac{1}{2}(\partial x_{3}+\partial x_{4}),\\ &R^{N}(\psi_{*}(H_1),\psi_{*}(H_3))\psi_{*}(H_3)=\frac{1}{{\sqrt{2}}}(\partial x_{6}-\partial x_{5}),\nonumber
	\end{align*}
	\begin{align}\label{kk1}
	&R^{N}(\psi_{*}(H_1),\psi_{*}(H_3))\psi_{*}(H_1)=\frac{1}{2}\partial x_{6},\\ 
	&\quad R^{N}(\psi_{*}(H_2),\psi_{*}(H_3))\psi_{*}(H_2)=(\frac{1}{\sqrt{2}}-1)\partial x_{6}, \nonumber 
	\end{align}
	\begin{align*}\nonumber
	&R^{N}(\psi_{*}(H_2),\psi_{*}(H_3))\psi_{*}(H_3)=-\frac{1}{2}(\partial x_{3}+\partial x_{4}),\\ &R^{N}(\psi_{*}(H_1),\psi_{*}(H_2))\psi_{*}(H_2)=\frac{1}{{2\sqrt{2}}}(\partial x_{1}+\partial x_{2}).
	\end{align*}
and
\[
S^{N}(\psi_{\ast}H_{i},\psi_{\ast}H_{j})=
\begin{bmatrix}
		-\frac{3}{2\sqrt{2}}&0&0\\ 
	0&-\frac{3}{2\sqrt{2}}&0\\ 
	0&0&-\frac{1}{\sqrt{2}}\\        
\end{bmatrix}
.
\]

\begin{equation}\label{k33}
R^{N}=Trace (S^{N})=-2{\sqrt{2}}.
\end{equation}

Again using (\ref{d3}), we obtain  $\lambda=\frac{3\alpha}{{2\sqrt{2}}}-\sqrt(2)\beta$ and $\mu =-\frac{\alpha}{2\sqrt{2}}$. Therefore ($(ker\psi_{\ast}^{\bot}),g)$ is admitting the expanding, shrinking and steady $\eta$-Ricci-Yamabe soliton according $\frac{53\alpha}{{2\sqrt{2}}}>\sqrt(2)\beta$,  $\frac{3\alpha}{{2\sqrt{2}}}<\sqrt(2)\beta$ or $\frac{3\alpha}{{2\sqrt{2}}}=\sqrt(2)\beta$ respectively.\\

Also, we have obatin the following cases for particular values of $\alpha$ and $\beta$ such as:\\
\noindent

{\bf {Case 1}}.\quad In an $\eta$-Ricci-Yamabe soliton of type $(\alpha,\beta)$ for particular value of  $\alpha=1, \beta=0$, we find $\lambda=\frac{3}{{2\sqrt{2}}}$ and $\mu=-\frac{1}{2\sqrt{2}}$, then ($(ker\psi_{\ast}^{\bot}),g)$ admitting expanding $\eta$-Ricci soliton.\\
\noindent

{\bf {Case 2}}.\quad In an $\eta$-Ricci-Yamabe soliton of type $(\alpha,\beta)$ for particular value of $\alpha=0, \beta=1$, we find $\lambda=-\sqrt{2}$ and $\mu=0$, then we have then ($(ker\psi_{\ast}^{\bot}),g)$ admitting a shrinking Yamabe soliton.\\

\section{Harmonic aspects of $\eta$-Ricci Yamabe soliton}
\indent
If the potential field $V$ is the gradient of some function $f$ on $M$, then $(M,g,V,\lambda,\mu,\alpha,\beta)$ is said to be a gradient $\eta$-Ricci Yamabe soliton which is denoted by $(M,g,V,\lambda,\mu,\alpha,\beta)$.\\

Now, consider the equation $\eta$-Ricci Yamabe soliton for $r$-dimensional fiber 
\begin{equation}\label{h1}
[\hat{g}(\hat{\nabla}_{E}V,F)+\hat{g}(\hat{\nabla}_{F}V,E)]+2\alpha \hat{S}(E,F)+(2\lambda-\beta\hat {R})\hat{g}(E,F)+2\mu\eta(E)\eta(F)=0. 
\end{equation}
Contracting the equation (\ref{h1}) we get 
\begin{equation}\label{h2}
div(V)=-r\lambda +\hat{R}\left(\frac{\beta}{2}-\alpha\right)-\mu.
\end{equation}
Thus, we have the following theorems:
\begin{theorem}\label{th1}
Let $(M,g,V,\lambda,\mu,\alpha,\beta)$ be an $\eta$-Ricci-Yamabe soliton of type $(\alpha,\beta)$ with vertical potential field $V$ and $\pi$ be a Riemannian submersion between Riemannian manifolds. If the vertical distribution $\mathcal{V}$ is parallel, then any fiber of Riemannian submersion $\pi$ is an $\eta$-Ricci-Yamabe soliton with $\eta$ be the $g$-dual $1$-form of the gradient potential field $V=grad(f)$, then the Laplacian equation satisfied by $f$ becomes
\begin{equation}\label{h3}
\Delta(f)=-r\lambda +\hat{R}\left(\frac{\beta}{2}-\alpha\right)-\mu.
\end{equation}
\end{theorem}
\begin{theorem}\label{th2}
Let $(M,g,V,\lambda,\mu,\beta)$ be an $\eta$-Ricci-Yamabe soliton of type $(0,\beta)$ with vertical potential field $V$ and $\pi$ be a Riemannian submersion between Riemannian manifolds. If the vertical distribution $\mathcal{V}$ is parallel, then any fiber of Riemannian submersion $\pi$ is an $\eta$-Yamabe soliton with $\eta$ be the $g$-dual $1$-form of the gradient potential field $V=grad(f)$, then the Laplacian equation satisfied by $f$ becomes
\begin{equation}\label{h4}
\Delta(f)=-r\lambda +\hat{R}\left(\frac{\beta}{2}\right)-\mu.
\end{equation}
\end{theorem}
\begin{remark}\label{hr1}
If $\mu=0$ in (\ref{h3}) and (\ref{h4}), we easily can obtain the similar type of results for the Ricci-Yamabe soliton and Yamabe soliton from Theorems (\ref{th1}) and (\ref{th2}), respectively.
\end{remark}
\noindent
Using the fact that a function $f:M\longrightarrow \mathbb{R}$ is said to be harmonic if $\Delta f=0$, where $\Delta$ is the Laplacian operator in $M$ \cite{Yau}, we have the following results:\\

\begin{theorem}\label{dt0}
Let  $(M,g,V,\lambda,\mu,\beta)$ be an $\eta$-Ricci-Yamabe soliton of type $(\alpha,\beta)$ with vertical potential field $V=grad(f)$ and $\pi$ be a Riemannian submersion between Riemannian manifolds. If the vertical distribution $\mathcal{V}$ is parallel, then any fiber of Riemannian submersion $\pi$  admitting gradient $\eta$-Ricci-Yamabe soliton with  potential harmonic function $f$ is expanding, steady and shrinking according as
	\begin{enumerate}
\item[(i)] \quad $\frac{\hat{R}}{r}\left(\frac{\beta}{2}-\alpha\right)>\frac{\mu}{r}$,
\item[(ii)]\quad $\frac{\hat{R}}{r}\left(\frac{\beta}{2}-\alpha\right)>\frac{\mu}{r}$ ,\quad and
\item[(iii)]\quad $\frac{\hat{R}}{r}\left(\frac{\beta}{2}-\alpha\right)>\frac{\mu}{r}$, \quad respectively.
\end{enumerate}
	\end{theorem}

\begin{corollary}\label{dt1}
Let  $(M,g,V,\lambda,\mu,\beta)$ be an $\eta$-Ricci-Yamabe soliton of type $(0,\beta)$ with vertical potential field $V=grad(f)$ and $\pi$ be a Riemannian submersion between Riemannian manifolds. If the vertical distribution $\mathcal{V}$ is parallel, then any fiber of Riemannian submersion $\pi$  admitting gradient $\eta$-Yamabe soliton with  potential harmonic function $f$ is expanding, steady and shrinking according as
	\begin{enumerate}
\item[(i)] \quad $\frac{\hat{R}}{r}\left(\frac{\beta}{2}\right)>\frac{\mu}{r}$,
\item[(ii)]\quad $\frac{\hat{R}}{r}\left(\frac{\beta}{2}\right)>\frac{\mu}{r}$ ,\quad and
\item[(iii)]\quad $\frac{\hat{R}}{r}\left(\frac{\beta}{2}\right)>\frac{\mu}{r}$, \quad respectively.
\end{enumerate}
	\end{corollary}
\indent
{\bf{Geometrical and Physical effects of  Ricci-Yamabe solitons.}} Geometry of  Ricci-Yamabe solitons, can develop a bridge between  a curvature inheritance symmetry of spacetime (semi-Riemannian manifold) and class of  Ricci-Yamabe solitons. In support of this affair we construct three mathematical models of conformally flat  Ricci-Yamabe solitons manifolds. As an application to relativity by investigating the kinematic and dynamic nature of spacetime, we present a physical models of three classes namely, shrinking, steady and expanding of perfect fluid solution of  Ricci-Yamabe solitons spacetime.\\
To deal with three special classes of  Ricci-Yamabe solitons, namely, shrinking ($\lambda<0$) which exists on a maximal time interval $-\infty<t<b$ where $b<\infty$, steady ($\lambda=0$) that which exists for all time or expanding $(\lambda>0)$  which exists on maximal time interval $a<t<\infty$, where $a>-\infty$. These classes yields example of {{\it ancient, eternal and immortal solution}} , respectively. Also, solutions of Einstein gravity coupled to a free massless scalar field with nonzero cosmological constant are  associated with shrinking or expanding  Ricci-Yamabe solitons.\\

{\bf{Significance of Laplacian equation in Physics.}} The general theory of solution of Laplace equation is known as {\textit{potential theory}} and the solution of Laplace equation are harmonic functions, which are important in branches of physics, electrostatics, gravitation and fluid dynamics. In modern physics, there are two fundamental forces of the nature known at the time, namely, gravity and the electrostatics forces, could be modeled using functions called the gravitational potential and electrostatics potential both of which satisfy Laplace equation. For example, consider the phenomena, if $\psi$ be the gravitational filed, $\rho$ the mass density and $G$ the gravitational constant. The Gauss's law of gravitational in differential form is 
\begin{equation}\label{l1}
\nabla\psi=-4\pi G\rho.
\end{equation} 
In case of  gravitational field, $\psi$ is conservative and can be expressed as the negative gradient of gravitational potential, i.e., $\psi=-grad (f)$ then by the Gauss's law of gravitational , we have
\begin{equation}\label{l2}
\nabla^{2}f=4\pi G\rho.
\end{equation}
 This physical phenomena is directly identical to the Theorems (\ref{th1}), (\ref{th2}) and equations (\ref{h3}) and (\ref{h4}), which is a Laplacian equation with potential vector filed of gradient type.\\

\end{document}